\documentclass[12pt]{amsart}
\usepackage{amssymb}
\usepackage[cmtip,all]{xy}

\usepackage{pgf, tikz}
\definecolor {processblue}{cmyk}{0.96,0,0,0}
\input xy
\xyoption{all}
\usepackage{mathtools}
\usepackage{amsfonts}
\usepackage{amsthm}

\usepackage{amscd}
\newtheorem{lemma}{Lemma}[section]
\newtheorem{corollary}[lemma]{Corollary}
\newtheorem{theorem}[lemma]{Theorem}
\newtheorem{proposition}[lemma]{Proposition}

\theoremstyle{definition}
\newtheorem{remark}[lemma]{Remark}
\newtheorem{definition}[lemma]{Definition}

\newcommand{\bsm}{\begin{smallmatrix}}
\newcommand{\esm}{\end{smallmatrix}}
\newcommand{\bbm}{\begin{matrix}}
\newcommand{\ebm}{\end{matrix}}

\begin{document}

\title{Representations of right 3-Nakayama  algebras}

\author{Alireza Nasr-Isfahani}
\address{Department of Mathematics\\
University of Isfahan\\
P.O. Box: 81746-73441, Isfahan, Iran\\ and School of Mathematics, Institute for Research in Fundamental Sciences (IPM), P.O. Box: 19395-5746, Tehran, Iran}
\email{nasr$_{-}$a@sci.ui.ac.ir / nasr@ipm.ir}
\author{Mohsen Shekari}
\address{Department of Mathematics\\
University of Isfahan\\
P.O. Box: 81746-73441, Isfahan, Iran}
\email{mshekari@sci.ui.ac.ir}



\begin{abstract}
In this paper we study the category of finitely generated modules over a right $3$-Nakayama artin algebra. First we give a characterization of right $3$-Nakayama artin algebras and then we give a complete list of non-isomorphic finitely generated indecomposable modules over any right $3$-Nakayama artin algebra. Also we compute all almost split sequences for the class of right $3$-Nakayama artin algebras. Finally, we classify finite dimensional right $3$-Nakayama algebras in terms of their quivers with relations.
\end{abstract}

\maketitle

\section{introduction}
Let $R$ be a commutative artinian ring. An $R$-algebra $\Lambda$ is called an artin algebra if
$\Lambda$ is finitely generated as a $R$-module. Let $\Lambda$ be an artin algebra. A right $\Lambda$-module $M$ is called uniserial ($1$-factor serial) if it has a unique composition series. An artin algebra $\Lambda$ is called Nakayama algebra if any indecomposable right $\Lambda$-module is uniserial. The class of Nakayama algebras is one the important class of representation finite algebras whose representation theory completely understood \cite{ARS}. According to \cite[Definition 2.1]{NS}, a non-uniserial right $\Lambda$-module $M$ of length $l$ is called $n$-factor serial ($l\geq n>1$), if $\frac{M}{\mathit{rad}^{l-n}(M)}$ is uniserial and $\frac{M}{\mathit{rad}^{l-n+1}(M)}$ is not uniserial. An artin algebra $\Lambda$ is called right $n$-Nakayama if every indecomposable right $\Lambda$-module is $i$-factor serial for some $1\leq i\leq n$ and there exists at least one indecomposable $n$-factor serial right $\Lambda$-module \cite[Definition 2.2]{NS}. The authors in \cite{NS} showed that the class of right $n$-Nakayama algebras provide a nice partition of the class of representation finite artin algebras. More precisely, the authors proved that an artin algebra $\Lambda$ is representation finite if and only if $\Lambda$ is right $n$-Nakayama for some positive integer $n$ \cite[Theorem 2.18]{NS}. The first part of this partition is the class of Nakayama algebras and the second part is the class of right $2$-Nakayama algebras. Indecomposable modules and almost split sequences for the class of right $2$-Nakayama algebras are classified in section $5$ of \cite{NS}. In this paper we will study the class of right $3$-Nakayama algebras.
We first show that an artin algebra $\Lambda$ which is neither Nakayama nor right $2$-Nakayama is right $3$-Nakayama if and only if every indecomposable right $\Lambda$-module of length greater than $4$ is uniserial and every indecomposable right $\Lambda$-module of length $4$ is local. Then we classify all indecomposable modules and almost split sequences over a right $3$-Nakayama artin algebra. We also show that finite dimensional right $3$-Nakayama algebras are special biserial and we describe all finite dimensional right $3$-Nakayama algebras by their quivers and relations. Riedtmann in \cite{R1} and \cite{R2}, by using the covering theory, classified representation-finite self-injective algebras. By elementary proof, avoiding many technical things, we classify self-injective special biserial algebras of finite type which able us to classify self-injective right $3$-Nakayama algebras.

The paper is organized as follows. In Section 2 we first study $3$-factor serial right modules and after classification of right $3$-Nakayama algebras we describe all indecomposable modules and almost split sequences over right $3$-Nakayama artin algebras.

In Section 3 we show that any finite dimensional right $3$-Nakayama algebra is special biserial and then we describe the structure of quivers and their relations of right $3$-Nakayama algebras.

In the final section, we classify self-injective right $3$-Nakayama algebras.

\subsection{notation }
Throughout this paper all modules are finitely generated right $\Lambda$-modules and all fields are algebraically closed fields unless otherwise stated. For a $\Lambda$-module $M$, we denote by $soc(M)$, $top(M)$, $rad(M)$, $\textit{l}(M)$, $\textit{ll}(M)$ and $\mathbf{dim} M$ its socle, top, radical, length, Loewy length and dimension vector, respectively. We also denote by $\tau(M)$, the Auslander-Reiten translation of $M$.
Let $Q=(Q_0, Q_1, s, t)$ be a quiver and $\alpha:i\rightarrow j$ be an arrow in $Q$. One introduces a formal inverse $\alpha^{-1}$ with $s({\alpha}^{-1}) = j$
and $t(\alpha^{-1}) = i$. An edge in $Q$ is an arrow or the inverse of an arrow. To each vertex $i$ in $Q$,
one associates a trivial path, also called trivial walk, $\varepsilon_i$ with $s(\varepsilon_i) = t(\varepsilon_i) = i$. A non-trivial
walk $w$ in $Q$ is a sequence $w=c_1c_2\cdots c_n$, where the $c_i$ are edges such that $t(c_i) = s(c_{i+1})$ for all $i$, whose inverse $w^{-1}$ is defined to be
the sequence $w^{-1}=c_n^{-1}c_{n-1}^{-1}\cdots c_1^{-1}$. A walk $w$ is called reduced if $c_{i+1}\neq c_i^{-1}$ for each $i$. For $i\in Q_0$, we denote by $i^+$ and $i^-$ the set of arrows starting in $i$ and the set of arrows ending in $i$, respectively, and for any set $X$, we denote by $\vert X\vert$ the number of elements in $X$.

\section{right $3$-Nakayama algebras}

In this section we give a characterization of right $3$-Nakayama artin algebras. We also classify all indecomposable modules and almost split sequences over right $3$-Nakayama algebras.

\begin{definition} \cite[Definitions 2.1, 2.2]{NS}
Let $\Lambda$ be an artin algebra and $M$ be a right $\Lambda$-module of length $l$.
\begin{itemize}
\item[$(1)$] $M$ is called  $1$-factor serial (uniserial) if $M$ has a unique composition series.
\item[$(2)$] Let $l\geq n>1$. $M$ is called $n$-factor serial if $\frac{M}{\mathit{rad}^{l-n}(M)}$ is uniserial and $\frac{M}{\mathit{rad}^{l-n+1}(M)}$ is not uniserial.
\item[$(3)$] $\Lambda$ is called right $n$-Nakayama if every indecomposable right $\Lambda$-module is $i$-factor serial for some $1 \leqslant i \leqslant n$ and there exists at least one  indecomposable $n$-factor serial right $\Lambda$-module.
\end{itemize}
\end{definition}

\begin{lemma} \label{l1}
Let $\Lambda$ be an artin algebra and $M$ be an indecomposable  right $\Lambda$-module of length  $r$ and Loewy length $t$. Then the following conditions are equivalent:
\begin{itemize}
\item[$(a)$] $M$ is a $3$-factor serial right $\Lambda$-module.
\item[$(b)$] One of the following conditions hold:
\begin{itemize}
\item[$(i)$] $M$ is local and for  every $1\leq i\leq r-4$, $ rad^{i}(M) $ is local and $ rad^{r-3}(M) $ is not local that either
\begin{itemize}
\item[$(1)$] $r=t+2$, $soc(M)=rad^{r-3}(M)$ and $\textit{l}(soc(M))=3$ or
\item[$(2)$] $r=t+1$ and $rad^{r-2}(M)$ is simple.
\end{itemize}
\item[$(ii)$] $M$ is not local, $r=3$ and $t=2$.
\end{itemize}
\end{itemize}
\end{lemma}

\begin{proof}
$(a)\Longrightarrow (b)$. Assume that $M$ is a local right $\Lambda$-module, then by \cite[Theorem 2.6]{NS}, for every $0\leq i \leq r-4$, $rad^{i}(M)$ is local and $ rad^{r-3}(M) $ is not local. On the other hand by \cite[Lemma 2.21]{NS}, $r\leq t+2$ and since $M$ is not uniserial $t<r$. If $r= t+2$, by \cite[Remark 2.7]{NS}, $soc(M)\subseteq rad^{r-3}(M) $ and since  $t=r-2$,   $soc(M)= rad^{r-3}(M) =S_{1}\oplus S_{2}\oplus S_{3}$. \\
If $r= t+1$, then  $rad^{r-2}(M)\neq 0$. If $l(rad^{r-2}(M))=2$, then $ \frac{M}{rad^{r-2}(M)} $ is uniserial which gives a contradiction. Thus $ rad^{r-2}(M) $ is simple, which complete the proof of $(i)$.
If $M$ is not local, then by \cite[Corollary 2.8]{NS} $r=3$, $t=2$ and the result follows. \\
$(b)\Longrightarrow (a)$. If  $M$ is not local and $r=3$,  then by \cite[Corollary 2.8]{NS}, $M$ is a $3$-factor serial right $\Lambda$-module. Now assume that  $M$ satisfies the condition $(i)$. Then $ \frac{M}{rad^{r-3}(M)} $ is uniserial . If $M$ satisfies the condition $(1)$, then $rad^{r-2}(M)=0$ and so $ \frac{M}{rad^{r-2}(M)}\cong M $ is not uniserial. If $M$ satisfies the condition $(2)$, then $\frac{M}{rad^{r-2}(M)}$ is non-uniserial and so $M$ is a $3$-factor serial right $\Lambda$-module.
\end{proof}

 \begin{lemma}\label{l2}
 Let $n> 1$ be a positive integer, $\Lambda$ be a right $n$-Nakayama artin algebra and $M$ be an indecomposable $n$-factor serial right $\Lambda$-module. Then $top(M)$ is simple if and only if $M$ is projective.
 \end{lemma}
 \begin{proof}
Assume that $top(M)$ is simple and $M$ is not projective. Let  $P\longrightarrow M$ be  a projective cover of $M$. Since $top(M)$ is simple, $P$ is indecomposable. Then by \cite[Lemma 2.11]{NS}, $P$ is a $t$-factor serial right $\Lambda$-module for some $t\geq n+1$ which gives a contradiction. Then $M$ is projective.
 \end{proof}

 \begin{theorem}\label{T1}
 Let $\Lambda$ be a right $3$-Nakayama artin algebra and $M$ be an indecomposable right $\Lambda$-module. Then $M$ is either a factor of an indecomposable projective right $\Lambda$-module or a submodule of an indecomposable injective right $\Lambda$-module.
 \end{theorem}
 \begin{proof}
If $M$ is uniserial or $2$-factor serial, then by definition of uniserial modules and  \cite[Lemma 5.2]{NS}, $M$ is a factor of an indecomposable projective module. Now assume that $M$ is  $3$-factor serial. If  $top(M)$ is simple, then by Lemma \ref{l2}, $M$ is projective. If $top(M)$ is not simple, then by Lemma \ref{l1}, $l(M)=3$ and $soc(M)$ is simple. This implies that $M$ is a submodule of an indecomposable injective right $\Lambda$-module.
 \end{proof}

An artin algebra $\Lambda$ is said to be of local-colocal type if every indecomposable right $\Lambda$-module is local or colocal (i.e. has a simple socle) \cite{T2}.

 \begin{corollary}\label{c1}
 Let $\Lambda$ be a right $3$-Nakayama artin algebra. Then $\Lambda$ is of local-colocal type.
 \end{corollary}

An artin algebra $\Lambda$ is said to be of right $n$-th local type if for every indecomposable right $\Lambda$-module $M$,
  $top^{n}(M)=\frac{M}{rad^{n}(M)}$ is indecomposable \cite{A1}.

  \begin{proposition}\label{p2}
  Let $\Lambda$ be a right $3$-Nakayama artin algebra and $M$ be an indecomposable right $\Lambda$-module. Then the following statements hold.
  \begin{itemize}
 \item[$(a)$] If $M$ is $2$-factor serial, then $l(M)=3$ and $ll(M)=2$.
  \item[$(b)$] If $M$ is $3$-factor serial and non-local, then $l(M)=3$ and $ll(M)=2$.
  \item[$(c)$]  If $M$ is $3$-factor serial and local, then $l(M)=4$ and $ll(M)=3$.
  \end{itemize}
  \end{proposition}
  \begin{proof}
 Any  indecomposable right $\Lambda$-module is either uniserial or $2$-factor serial or $3$-factor serial, then  by definition of uniserial module,  \cite[Lemma 5.2]{NS} and Lemma \ref{l1}, $top^{2}(M)=\frac{M}{rad^{2}(M)}$ is indecomposable and so $\Lambda$ is of right $2$-nd local type. Let $M$ be a $2$-factor serial right $\Lambda$-module, so by \cite[Lemma 1.4]{A1}, $rad(M)$ is not local and by  \cite[Corollary 2.8]{NS}, $l(rad(M))=2$. This proves part $(a)$. \\
The part $(b)$ follows from Lemma \ref{l1}. \\
Let $M$ be a local $3$-factor serial right $\Lambda$-module. By \cite[Lemma 1.4]{A1}, $rad(M)$ is not local and by \cite[Theorem 2.6]{NS}, $l(M)=4$. By Lemma \ref{l1}, $ll(M)=3$ and the result follows.
  \end{proof}

In the next theorem, we give a characterization of right $3$-Nakayama artin algebras.

  \begin{theorem}\label{p3}
  Let $\Lambda$ be an artin algebra which is  neither Nakayama nor right $2$-Nakayama. Then $\Lambda$ is right $3$-Nakayama if and only if every indecomposable right $\Lambda$-module $M$ with   $l(M)>4$ is uniserial and every indecomposable right $\Lambda$-module $M$ with  $l(M)=4$ is local.
  \end{theorem}
  \begin{proof}
  Assume that $\Lambda$ is a right $3$-Nakayama algebra. It follows from Proposition \ref{p2} that, every indecomposable right $\Lambda$-module $M$ with $l(M)>4$ is uniserial.  Assume that there exists an indecomposable right $\Lambda$-module $M$ with  $l(M)=4$ which is not local. Then  by  \cite[Corollary 2.8]{NS}, $M$ is $4$-factor serial which is a contradiction. Conversely, assume that any indecomposable right $\Lambda$-module $M$ with $l(M)>4$ is uniserial and every indecomposable right $\Lambda$-module $M$ with  $l(M)=4$ is local. Since  $\Lambda$ is neither Nakayama nor right $2$-Nakayama, there exists an indecomposable $t$-factor serial right $\Lambda$-module $M$ for some $t\geq 3$. Also for any indecomposable right $\Lambda$-module $N$ of length $4$, $N$ is local and by \cite[Corollary 2.8]{NS}, $N$ is $r$-factor serial right $\Lambda$-module for some $r\leq 3$. Therefor $M$ is $3$-factor serial and $\Lambda$ is right $3$-Nakayama.
 \end{proof}

  \begin{corollary}\label{c3}
  Let $\Lambda$ be an artin algebra. Then the following statements hold.
  \item[$(a)$] If every indecomposable right $\Lambda$-module of length greater than $4$ is  uniserial and every indecomposable right $\Lambda$-module of length $4$ is local, then $\Lambda$ is either Nakayama or right $2$-Nakayama or right $3$-Nakayama.
  \item[$(b)$]  If every indecomposable right $\Lambda$-module of length greater than $4$ is  uniserial, every indecomposable right $\Lambda$-module of length $4$ is local and there exists an indecomposable non-uniserial right $\Lambda$-module which is not projective, then $\Lambda$ is right $3$-Nakayama.
  \end{corollary}

\begin{proof}
$(b)$ It is clear that $\Lambda$ is not Nakayama. If $\Lambda$ is a right $2$-Nakayama, then by  \cite[Proposition 5.5]{NS} every indecomposable non-uniserial right $\Lambda$ is projective, which gives a contradiction. Thus by Theorem \ref{p3}, $\Lambda$ is right $3$-Nakayama.
 \end{proof}

  \begin{lemma}\label{l3}
Let $\Lambda$ be a right $3$-Nakayama artin algebra and $M$ be an indecomposable $3$-factor serial right $\Lambda$-module. Then $l(soc(M))\leq 2$.
\end{lemma}
\begin{proof}
By Proposition \ref{p2}, $l(M)$ is either $3$ or $4$. If $l(M)=3$, then by Lemma \ref{l1}, $soc(M)$ is simple. Now assume
that $l(M)=4$. Then by Proposition \ref{p2} and Lemma \ref{l2}, $M$ is projective. If $M$ is injective, then $soc(M)$ is simple. Assume that $M$ is not injective. Assume on a contrary that  $soc(M)=S_1\oplus S_2\oplus S_3$, for simple $\Lambda$-modules $S_1, S_2$ and $S_3$. Then $rad(M)=Soc(M)$ and we have a right minimal almost split morphism $S_1\oplus S_2\oplus S_3\longrightarrow M$. Thus $S_i$ is not injective for each $1\leq i\leq 3$ and we have the following almost split sequences
\begin{center}
$0\longrightarrow S_1\longrightarrow M\longrightarrow \tau^{-1}(S_1)\longrightarrow0 $
\end{center}
\begin{center}
$0\longrightarrow S_2\longrightarrow M\longrightarrow \tau^{-1}(S_2)\longrightarrow0 $
\end{center}
\begin{center}
$0\longrightarrow S_3\longrightarrow M\longrightarrow \tau^{-1}(S_3)\longrightarrow0 $
\end{center}
\begin{center}
$0\longrightarrow M\longrightarrow {\tau^{-1}(S_1)\oplus \tau^{-1}(S_2)\oplus \tau^{-1}(S_3)}\buildrel{ [f_{1}, f_{2}, f_{3}]}\over\longrightarrow \tau^{-1}(M)\longrightarrow 0$
\end{center}

Then $\tau^{-1}(M)$ is an indecomposable of length $5$ and by Theorem \ref{p3}, $\tau^{-1}(M)$ is uniserial. On the other hand for any $1\leq i \leq3$, the irreducible morphism $ f_i:\tau^{-1}(S_i)\longrightarrow \tau^{-1}(M)$ is a monomorphism and so $\tau^{-1}(S_i)$ is uniserial. Also by \cite[Theorem 2.13]{NS}, there exists $1\leq i \leq 3$ such that $\tau^{-1}(S_i)\cong \frac{M}{S_i}$ is not uniserial, which gives a contradiction. Therefor $l(Soc(M))\leq 2$ and the result follows.
 \end{proof}

 \begin{theorem}\label{T11}
 Let $\Lambda$ be a right $3$-Nakayama artin algebra and $M$ be an indecomposable right $\Lambda$-module. Then the following statements hold.
 \item[$(a)$] Assume that $M$ is  $2$-factor serial. Then submodules of $M$ are $S_1, S_2$ and $rad(M)=soc(M)=S_1\oplus S_2$, where $S_i$ is a simple submodule of $M$ for each $i=1, 2$.
 \item[$(b)$] Assume that $M$ is local-colocal $3$-factor serial. Then submodules of  $M$ are $rad(M)$ which is indecomposable non-local $3$-factor serial of length $3$, two uniserial submodules $M_1$ and $M_2$ of length $2$ and $S:=soc(M)$ which is simple and $ll(M)=3$.
 \item[$(c)$] Assume that $M$ is local and non-colocal $3$-factor serial. Then submodules of $M$ are uniserial submodule $N$ of length $2$, simple submodules $S$ and $S'$ that $soc(N)=S'$,   $rad(M)=N\oplus S$ and $soc(M)=S^{'}\oplus S$ and $ll(M)=3$.
 \item[$(d)$] Assume that $M$ is colocal and non-local $3$-factor serial. Then submodules of $M$ are two uniserial modules $M_1$ and $M_2$ of length $2$ and $rad(M)=soc(M)=S$ which is simple and $ll(M)=2$.
 \end{theorem}

\begin{proof}
$(a)$ and $(c)$ follow from Proposition \ref{p2} and Lemma \ref{l3}.\\
$(b).$ Since $M$ is local, $rad(M)$ is  maximal submodule of $M$ and  $soc(M)\subseteq rad(M)$ and since $soc(M)$ is simple, by \cite[Theorem 2.6]{NS} and Proposition  \ref{p2},  $rad(M)$ is non-local indecomposable right $\Lambda$-module of length $3$. Therefore $rad(M)$ has two uniserial maximal submodules $M_1$ and $M_2$ of length $2$.\\
The proof of $(d)$ is similar to the proof of $(b)$.
\end{proof}

  \begin{proposition}\label{p9}
  Let $\Lambda$ be a right $3$-Nakayama artin algebra and $M$ be an indecomposable right $\Lambda$-module. Then $M$ is not projective if and only if one of the following situations holds.
  \item[$(a)$] $M$ is local and  there is an indecomposable projective right $\Lambda$-module $P$ such that $M$ is a factor of $P$ where $P$ satisfy one of the following situations:
  \begin{itemize}

 \item[$(i)$] $P$ is an uniserial projective right $\Lambda$-module. So $M\cong \frac{P}{rad^{i}(P)}$ for some $1\leq i<l(P)$.
 \item[$(ii)$] $P$ is a $2$-factor serial projective right $\Lambda$-module  that
 $rad(P)=soc(P)=S_1\oplus S_2$ where  $S_i$ is simple submodule for each $1\leq i \leq 2$. So $M$ is isomorphic to either  $\frac{P}{rad(P)}$ or $ \frac{P}{S_i}$ for some $1\leq i\leq 2$.
 \item[$(iii)$] $P$ is a $3$-factor serial projective-injective right $\Lambda$-module  and submodules of $P$ are  $rad(P)$ which is indecomposable non-local $3$-factor serial right $\Lambda$-module of length $3$, two uniserial modules $M_1$ and $M_2$ of length $2$ and $S=soc(P)$ which  is simple. Then $M$ is isomorphic to either $\frac{P}{rad(P)}$ or $\frac{P}{M_i}$ for some $1\leq i\leq 2$ or $\frac{P}{S}$.
 \item[$(iv)$] $P$ is a $3$-factor serial projective non-injective right $\Lambda$-module.  $rad(P)=N\oplus S$ where $N$ is an uniserial submodule  of length $2$ and $S$ is a simple submodule of $P$, $soc(P)=S^{'}\oplus S$ where  $S^{'}=soc(N)$. So $M$ is isomorphic to either $ \frac{P}{rad (P)} $ or $ \frac{P}{N} $ or $\frac{P}{soc(P)}$ or  $\frac{P}{S}$ or $\frac{P}{S^{'}}$.

  \end{itemize}
 \item[$(b)$] $M$ is non-local $3$-factor serial of length $3$ where submodules of $M$ are two uniserial modules $M_1$ and $M_2$ of length $2$ and $rad(M)=soc(M)=S$ which is simple. $M$ is either injective or a submodule of $3$-factor serial projective-injective indecomposable module.
  \end{proposition}
\begin{proof}
It follows from  Corollary \ref{c1}, Theorem \ref{T11} and \cite[Lemma 5.3]{NS}.
\end{proof}

Now we characterize almost split sequences of right $3$-Nakayama algebras.

  \begin{theorem}\label{T3}
  Let $\Lambda$ be a right $3$-Nakayama artin algebra and $M$ be an indecomposable non-projective right $\Lambda$-module. Then one of the following situations hold:
  \begin{itemize}
\item[$(A)$]  Assume that $M\cong \frac{P}{rad^i(P)}$  where $P$ is uniserial projective for some $1\leq i <l(P)$.
\begin{itemize}

\item[$(i)$] If $M$ is a simple direct summand of $top(L)$, where $L$ is an indecomposable non-local, submodules of $L$ are including $M_1$ and $M_2$ that are uniserial modules of length $2$ and $soc(L)$ which is a simple module. In this case $M\cong \frac{L}{M_j}$ for some $1\leq j \leq 2$ and the following exact sequence
\begin{center}
$0\longrightarrow M_j \buildrel{ i_{1} }\over\longrightarrow L \buildrel{ \pi_{1} }\over\longrightarrow \frac{L}{M_ {j}}\longrightarrow0  $
\end{center}
is almost split sequence for each $j=1, 2$.
\item[$(ii)$] Otherwise, the sequence
  \begin{center}
  $0\longrightarrow \frac{rad (P)}{rad^{i+1}(P)} \buildrel{
  \begin{bmatrix}
\pi_2\\
i_2
  \end{bmatrix}
   }\over\longrightarrow \frac{rad (P)}{rad^{i}(P)}\oplus \frac{P}{rad^{i+1}(P)} \buildrel{[-i_{3},\pi_3]}\over\longrightarrow \frac{P}{rad^{i}(P)}\longrightarrow 0$
  \end{center}
    is an almost split sequence.

\end{itemize}
 \item[$(B)$] Assume that $M$ is a factor of $2$-factor serial projective right $\Lambda$-module $P$, where $rad(P)=soc(P)=S_1\oplus S_2$ that $S_1$ and $S_2$ are simple modules.
 \begin{itemize}
 \item[$(i)$]If $M\cong \frac{P}{S_i}$ for some $1\leq i \leq 2$, then the sequence
 \begin{center}
 $0\longrightarrow S_i \buildrel{i_4}\over \longrightarrow P\buildrel{\pi_4}\over \longrightarrow \frac{P}{S_i}\longrightarrow0$
 \end{center}
 is an almost split sequence.

 \item[$(ii)$] If $M\cong \frac{P}{rad (P)}$, then the sequence
 \begin{center}
 $0\longrightarrow P\buildrel{
  \begin{bmatrix}
\pi_5\\
\pi_6
  \end{bmatrix}
   }\over \longrightarrow \frac{P}{S_1}\oplus \frac{P}{S_2}\buildrel{[-\pi_7,\pi_8]}\over \longrightarrow \frac{P}{rad (P)}\longrightarrow0$
  \end{center}
 is an almost split sequence.
 \end{itemize}
 \item[$(C)$] Assume that $M$ is a factor of $3$-factor serial projective-injective right $\Lambda$-module $P$ and submodules of $P$ are $rad(P)$ which is indecomposable non-local $3$-factor serial of length $3$, two uniserial modules $M_1$ and $M_2$ of length $2$ and $S=soc(P)$ that is simple.
 \begin{itemize}
\item[$(i)$] If $M\cong \frac{P}{rad (P)}$, then the sequence
 \begin{center}
 $ 0\longrightarrow\frac{P}{S}\buildrel{
  \begin{bmatrix}
\pi_9\\
\pi_{10}
  \end{bmatrix}
   }\over \longrightarrow \frac{P}{M_1}\oplus\frac{P}{M_2}\buildrel{[-\pi_{11},\pi_{12}]}\over \longrightarrow \frac{P}{rad (P)}\longrightarrow0 $
 \end{center}
 is an almost split sequence.
 \item[$(ii)$]
 If $M\cong \frac{P}{M_i}$ for some  $1\leq i\leq 2$, then the sequence
 \begin{center}
 $ 0\longrightarrow\frac{M_i}{S}\buildrel{i_5}\over \longrightarrow \frac{P}{S}\buildrel{\pi_{13}}\over \longrightarrow \frac{P}{M_i}\longrightarrow0 $
 \end{center}
 is an almost split sequence.

 \item[$(iii)$]
  If $M\cong \frac{P}{S}$, then the sequence
 \begin{center}
 $ 0\longrightarrow rad (P)\buildrel{
  \begin{bmatrix}
\pi_{14}\\
 i_6
  \end{bmatrix}
   }\over \longrightarrow\frac{rad (P)}{S}\oplus P\buildrel{[-i_7,\pi_{15}]}\over \longrightarrow \frac{P}{S}\longrightarrow0 $
 \end{center}
 is an almost split sequence.
  \end{itemize}
 \item[$(D)$] Assume that $M$ is a factor of $3$-factor serial  non-injective  projective right $\Lambda$-module $P$. That   $rad(P)=N\oplus S$ where $N$ is an  uniserial submodule  of length $2$ and $S$ is a simple submodule of $P$ and $soc(P)=S^{'}\oplus S$ where $S^{'}=soc(N)$.
\begin{itemize}
\item[$(i)$] If $M\cong \frac{P}{rad (P)}$, then the sequence
\begin{center}
$0\longrightarrow \frac{P}{S^{'}}\buildrel{
  \begin{bmatrix}
\pi_{16}\\
\pi_{17}
  \end{bmatrix}   }\over\longrightarrow\frac{P}{N}\oplus \frac{P}{soc(P)}\buildrel{[-\pi_{18},\pi_{19}] }\over\longrightarrow \frac{P}{rad (P)}\longrightarrow0$
\end{center}
is an almost split sequence.
\item[$(ii)$] If $M\cong \frac{P}{N}$, then the sequence
\begin{center}
$ 0\longrightarrow \frac{N}{S^{'}}\buildrel{i_{8} }\over \longrightarrow \frac{P}{S^{'}}\buildrel{\pi_{20} }\over \longrightarrow \frac{P}{N}\longrightarrow 0$
\end{center}
is an almost split sequence.
\item[$(iii)$] If $M\cong \frac{P}{soc(P)}$, then the sequence
\begin{center}
$  0\longrightarrow P\buildrel{
  \begin{bmatrix}
\pi_{21}\\
\pi_{22}
  \end{bmatrix}
   }\over\longrightarrow \frac{P}{S}\oplus \frac{P}{S^{'}}\buildrel{[-\pi_{23},\pi_{24}]}\over\longrightarrow \frac{P}{soc(P)}\longrightarrow0$
\end{center}
is an almost split sequence.
\item[$(iv)$]  If $M\cong \frac{P}{S}$, then the sequence
\begin{center}
$ 0\longrightarrow S\buildrel{i_{9} }\over \longrightarrow P\buildrel{\pi_{25} }\over \longrightarrow \frac{P}{S}\longrightarrow 0$
\end{center}
is an almost split sequence.
\item[$(v)$] If $M\cong \frac{P}{S^{'}}$, then the sequence
\begin{center}
$ 0\longrightarrow N\buildrel{  \begin{bmatrix}
\pi_{26}\\
 i_{10}
  \end{bmatrix}}\over \longrightarrow \frac{N}{S^{'}}\oplus P\buildrel{[-i_{11},\pi_{27}]}\over \longrightarrow \frac{P}{S^{'}}\longrightarrow 0$
\end{center}
\end{itemize}
\item[$(E)$] Assume  that $M$ is a non-local $3$-factor serial right $\Lambda$-module of  length $3$ and submodules of $M$ are two uniserial maximal submodules $M_1$ and $M_2$ of length $2$ and $rad(M)=soc(M)=S$ which is simple.

Then the following exact sequence
\begin{center}
$ 0\longrightarrow S\buildrel{  \begin{bmatrix}
i_{12}\\
i_{13}
  \end{bmatrix}}\over\longrightarrow M_{1}\oplus M_{2}\buildrel{ [-i_{14}, i_{15}]}\over \longrightarrow M\longrightarrow 0$
\end{center}
is an almost split sequence.
\\
Where $i_j $ is an inclusion for each $1\leq j \leq 15$ and $\pi_j $ is a canonical epimorphism for each $1\leq j \leq 27$.

\end{itemize}
  \end{theorem}
\begin{proof}
Put $g_1=[-i_3, \pi_3]$, $g_2=[-\pi_7,\pi_8]$, $g_3=[-\pi_{11}, \pi_{12}]$, $g_4=[-i_7,\pi_{15}]$, $g_5=[-\pi_{18}, \pi_{19}]$, $g_6=[-\pi_{23}, \pi_{24}]$, $g_7=[-i_{11},\pi_{27}]$ and $g_8= [-i_{14},i_{15}]$. It is easy to see that all given sequences are exact, non-split and have indecomposable end terms. It is enough to show that homomorphisms $\pi_1$, $g_1$, $\pi_4$,
  $ g_2 $, $g_3$, $ \pi_{13} $, $ g_4 $, $ g_5 $, $ \pi_{20} $, $ g_6 $, $ \pi_{25} $, $g_7$ and $g_8$ are right almost split morphisms.
\item[$(A)(i)$] Let  $V$ be an indecomposable right $\Lambda$-module and $\nu:V\longrightarrow M$ be a non-isomorphism.  Since $M$ is a simple module, so $\nu$ is an epimorphism. If $j=1$, then $V$ is isomorphic to either $M_2$ or $\frac{P}{rad^{i}(P)}$ for some $2\leq i < l(P)$. Since  $M_2$ is a submodule of $L$ and $top(\frac{P}{rad^{i}(P)})$ is a direct summand of $top(L)$ for each $2\leq i < l(P)$, there exists a homomorphism $h: V\longrightarrow L$ such that $\pi_1 h=\nu$. In case $j=2$ the proof is similar.
 \item[$(A)(ii)$] Let $V$ be an indecomposable right $\Lambda$-module and $\nu:V\longrightarrow \frac{P}{rad^{i}(P)}$ be a non-isomorphism. If $\nu$ is an epimorphism, then  $V\cong \frac{P}{rad^{s}(P)}$ for some $s>i$. This implies that, there is a homomorphism $h:V\longrightarrow \frac{rad(P)}{rad^{i}(P)}\oplus \frac{P}{rad^{i+1}(P)}$ such that $ \nu=g_1h $. Now assume that $\nu$ is not an epimorphism, then Im$(\nu)= \frac{rad^{t}(P)}{rad^{i}(P)}$ for some $t<i$. Then there is a homomorphism $h:V\longrightarrow \frac{rad(P)}{rad^{i}(P)}\oplus \frac{P}{rad^{i+1}(P)}$ such that $ \nu=g_1h $.
\item[$(B)(i)$] Let $V$ be an indecomposable right $\Lambda$-module and $\nu:V\longrightarrow \frac{P}{S_i}$ be a non-isomorphism. If $\nu$ is an  epimorphism, then $V\cong P$. This implies that there is an isomorphism $h:V\longrightarrow P$ such that $ \nu=\pi_{4}h $. Now assume that $\nu$ is not an epimorphism. Since $l(P)=3$,  $l(\frac{P}{S_i})=2$ and so Im$(\nu)$ is a simple submodule of $M$ which is  isomorphism to the direct summand of $soc(P)$. Then there is a homomorphism $h:V\longrightarrow P$ such that $ \nu=\pi_{4}h $.
\item[$(B)(ii)$] Let $V$ be an indecomposable right $\Lambda$-module and $\nu:V\longrightarrow \frac{P}{rad(P)}$ be a non-isomorphism. Since $\frac{P}{rad(P)}$ is simple,  $\nu$ is an epimorphism and $V$ is isomorphic to either $\frac{P}{S_i}$ for some $1\leq i\leq 2$ or $P$. So there is a homomorphism $h:V\longrightarrow \frac{P}{S_1}\oplus \frac{P}{S_2}$ such that $g_2 h=\nu$.
\item[$(C)(i)$] The proof is similar to the proof of the $B(ii)$.
\item[$(C)(ii)$] Let $V$ be an indecomposable right $\Lambda$-module and $\nu:V\longrightarrow \frac{P}{M_i}$ be a non-isomorphism. If $\nu$ is an epimorphism, then $V$ is isomorphic to either $ \frac{P}{S} $ or $P$. So there is a homomorphism $h:V\longrightarrow \frac{P}{S}$ such that $ \nu=\pi_{13}h $. Now assume that $\nu$ is not an  epimorphism. Then Im$(\nu)$ is simple and isomorphic to the direct summand of  $soc(\frac{P}{S}) $. This implies that there is a homomorphism $h:V\longrightarrow \frac{P}{S}$ such that $ \nu=\pi_{13}h $.
\item[$(C)(iii)$] Let $V$ be an indecomposable right $\Lambda$-module and $\nu:V\longrightarrow \frac{P}{S}$ be a non-isomorphism. If $\nu$ is an epimorphism, then $V\cong P$. This implies that there is a homomorphism $h:V\longrightarrow \frac{rad(P)}{S}\oplus P$ such that $ \nu=g_4 h $. Now assume that  $\nu$ is not an  epimorphism. Then Im$(\nu)$ is a submodule of $\frac{rad(P)}{S}$ and so there is a homomorphism $h:V\longrightarrow\frac{rad(P)}{S}\oplus P$ such that $ \nu=g_4 h $.
\item[$(D)(i)$] The proof is similar to the proof of the $B(ii)$.
\item[$(D)(ii)$] Let $V$ be an indecomposable right $\Lambda$-module and $\nu:V\longrightarrow \frac{P}{N}$ be a non-isomorphism. If $\nu$ is an epimorphism, then $V$ is isomorphic to either $ P$ or $ \frac{P}{S^{'}}$. So there is a homomorphism  $h:V\longrightarrow \frac{P}{S^{'}}$ such that $\pi_{20}h=\nu$.  Now assume that $\nu$ is not an epimorphism, so Im$(\nu)\cong S$ and $S$ is a direct summand of $ \frac{rad(P)}{S^{'}}$. This implies that there is a homomorphism  $h:V\longrightarrow \frac{P}{S^{'}}$ such that $\pi_{20}h=\nu$.
 \item[$(D)(iii)$] Let $V$ be an indecomposable right $\Lambda$-module and $\nu:V\longrightarrow \frac{P}{soc(P)}$ be a non-isomorphism. If $\nu$ is  an epimorphism, then $V$ is isomorphic to either $P$ or $ \frac{P}{S^{'}}$ or $ \frac{P}{S}$. So there is a homomorphism  $h:V\longrightarrow \frac{P}{S^{'}}\oplus \frac{P}{S}$ such that $g_{6}h=\nu$. If $\nu$ is not an epimorphism, then Im$(\nu)$ is simple and submodule of $ \frac{rad(P)}{soc(P)}$. This implies that, there is a homomorphism  $h:V\longrightarrow \frac{P}{S^{'}}\oplus \frac{P}{S}$ such that $g_{6}h=\nu$.
 \item[$(D)(iv)$] Let $V$ be an indecomposable right $\Lambda$-module and $\nu:V\longrightarrow \frac{P}{S}$ be a non-isomorphism. If $\nu$ is an epimorphism, then $V\cong P$. So there is a homomorphism  $h:V\longrightarrow P$ such that $\pi_{25}h=\nu$. If $\nu$ is not an epimorphism, then Im$(\nu)$ is a submodule of $\frac{rad (P)}{S}$ and $\frac{rad (P)}{S}\cong N$. This implies that there is a homomorphism  $h:V\longrightarrow P$ such that $\pi_{25}h=\nu$.
 \item[$(D)(v)$] Let $V$ be an indecomposable right $\Lambda$-module and $\nu:V\longrightarrow \frac{P}{S^{'}}$ be a non-isomorphism. If $\nu$ is an epimorphism, then $V\cong  P$ and there is a homomorphism  $h:V\longrightarrow \frac{N}{S^{'}}\oplus P$ such that $g_{7}h=\nu$. If $\nu$ is not an  epimorphism, then $Im(\nu)$ is a submodule of $\frac{P}{S^{'}}$. Therefor $\frac{P}{S^{'}}$ is a submodule of $ \frac{rad(P)}{S^{'}}=soc(\frac{P}{S^{'}})\cong S\oplus \frac{N}{S^{'}}$. This implies that there is a homomorphism $h:V\longrightarrow \frac{N}{S^{'}}\oplus P$ such that $g_{7} h=\nu$.
 \item[$(E)$] Let $V$ be an indecomposable right $\Lambda$-module and $\nu:V\longrightarrow M$ be a non-isomorphism.  Since $\nu$ is not an isomorphism and $M$ is not local, by Theorem \ref{T11}, $\nu$ is not epimorphism. Therefore $Im(\nu)$ is a submodule of $M$ and there is a homomorphism $ h: V\longrightarrow M_{1}\oplus M_{2} $ such that $g_8  h=\nu$.
 \end{proof}

\section{quivers of right $3$-Nakayama algebras}

In this section we describe finite dimensional right $3$-Nakayama algebras in terms of their quivers with relations.

 A finite dimensional $K$-algebra $\Lambda=\frac{KQ}{I}$ is called special biserial algebra provided $(Q,I)$ satisfying the following conditions:
\begin{itemize}
\item[$(1)$] For any vertex $a\in Q_0$, $|a^+|\leq 2$ and   $|a^-|\leq 2$.
\item[$(2)$] For any arrow $\alpha \in Q_1$, there is at most one arrow $\beta$ and at most one arrow $\gamma$ such that $\alpha\beta$ and $\gamma\alpha$ are not in $I$.
\end{itemize}
Let $\Lambda=\frac{KQ}{I}$ be a special biserial finite dimensional $K$-algebra. A walk $w=c_1c_2\cdots c_n$ in $Q$ is called string of length $n$ if $c_i\neq c_{i+1}^{-1}$ for each $i$ and no subwalk of $w$ nor its inverse is in $I$. In addition, we have strings of length zero, for any $a\in Q_0$ we have two strings of length zero, denoted by $1_{(a,1)}$ and $1_{(a,-1)}$. We have $s(1_{(a,1)})=t(1_{(a,1)})=s(1_{(a,-1)})=t(1_{(a,-1)})=a$ and $1_{(a,1)}^{-1}=1_{(a,-1)}$. A string $w=c_1c_2\cdots c_n$ with $s(w)=t(w)$ such that each
power $w^m$ is a string, but $w$ itself is not a proper power of any strings is called band. We denote by $\mathcal{S}(\Lambda)$ and $\mathcal{B}(\Lambda)$ the set of all strings of $\Lambda$ and the set of all bands of $\Lambda$, respectively. Let $\rho$ be the equivalence relation on $\mathcal{S}(\Lambda)$ which identifies every string $w$ with its inverse $w^{-1}$ and $\sigma$ be the equivalence relation on $\mathcal{B}(\Lambda)$ which identifies every band $w=c_1c_2\cdots c_n$ with the cyclically permuted bands $w_{(i)}=c_ic_{i+1}\cdots c_nc_1\cdots c_{i-1}$ and their inverses $w_{(i)}^{-1}$, for each $i$.
Butler and Ringel in \cite{BR} for each string $w$ defined a unique string module $M(w)$ and for each band $v$ defined a family of band modules $M(v,m,\varphi)$ with $m\geq 1$ and $\varphi\in Aut(K^m)$. Let $\widetilde{\mathcal{S}}(\Lambda)$ be the complete set of representatives of strings relative to $\rho$ and $\widetilde{\mathcal{B}}(\Lambda)$  be the complete set of representatives of bands relative to $\sigma$. Butler and Ringel in \cite{BR} proved that, the modules $M(w)$, $w\in \widetilde{\mathcal{S}}(\Lambda)$ and the modules $M(v,m,\varphi)$ with $v\in \widetilde{\mathcal{B}}(\Lambda)$, $m\geq 1$ and $\varphi\in Aut(K^m)$ provide complete list of pairwise non-isomorphic indecomposable $\Lambda$-modules.
Indecomposable $\Lambda$-modules are either string modules or band modules or non-uniserial projective-injective modules (see \cite{BR} and \cite{WW}). If $\Lambda$ is a special biserial algebra of finite type, then any indecomposable $\Lambda$-module is either string module or non-uniserial projective-injective module.

\begin{remark}\label{r1}
Let $Q$ be a finite quiver, $I$ be an admissible ideal of $Q$, $Q^{'}$ be a subquiver of $Q$ and $I^{'}$ be an admissible ideal of $Q^{'}$ which is restriction of $I$ to $Q^{'}$. Then there exists a fully faithful embedding
$F:rep_K(Q^{'}, I^{'})\longrightarrow rep_K(Q, I) $
\end{remark}

\begin{proposition}\label{p1} Any basic connected finite dimensional right $3$-Nakayama $K$-algebra is a special biserial algebra of finite type.
\end{proposition}
\begin{proof} Let $\Lambda=KQ/I$ be a right $3$-Nakayama algebra. By Theorem \cite[Theorem 2.18]{NS}, $\Lambda$ is of finite type. We show that for every $a\in Q_0$, $|a^+|\leq 2$.
If there exists a vertex $a$ of $Q_0$ such that $|a^+|\geq 3$, then we have two cases.
\begin{itemize}
\item Case $1$: The algebra $\Lambda_1=KQ_1$ given by the quiver $Q_1$
$$\hskip .5cm 
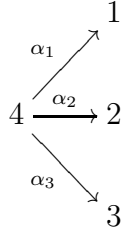
\xymatrix{
&{1}\ar @{<-}[dl]_{\alpha_{1}}\\
{4}\ar[r]^{\alpha_{2}}&{2}\\
&{3}\ar @{<-}[ul]^{\alpha_{3}}}\hskip .5cm$$
which is a subquiver of $Q$, is a subalgebra of $\Lambda$. There is  an indecomposable representation $M$ of $Q_1$ such that
  $\mathbf{dim}M=[1,1,1,2]^t$. $M$ is not local and by \cite[proposition 2.8]{NS},  $M$ is a $5$-factor serial right $\Lambda_1$-module. Therefore by using Remark \ref{r1}, there is a $5$-factor serial right $\Lambda$-module which is a contradiction.
\item Case $2$: The algebra $\Lambda_2=\frac{KQ_2}{I_2}$ given by the quiver $Q_2$
$$
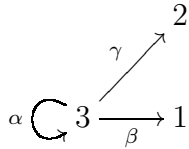
\begin{matrix}\xymatrix{
&{2}\ar @{<-}[dl]_{\gamma}\\
{3}\ar@(ul,dl)_{\alpha} \ar[r]_{\beta}&{1}}
\end{matrix}\hskip.5cm$$
which is a subquiver of $Q$ and the ideal $I_2$ which is a restriction of $I$ to $Q_2$, is a subalgebra of $\Lambda$. There is an indecomposable representation $M$ of $(Q_2,I_2)$ such that $\mathbf{dim}M=[1,1,3]^t$. $M$ is not local and by \cite[proposition 2.8]{NS}, $M$ is a $5$-factor serial right $\Lambda_2$-module. Then, there is a $5$-factor serial right $\Lambda$-module which is a contradiction.
\end{itemize}
Now we show that for every $a\in Q_0$, $|a^-|\leq 2$. Assume that there exists a vertex $a$ of $Q_0$ such that $|a^-|\geq 3$, then we have two cases.
\begin{itemize}
\item Case $1$: The algebra $\Lambda_1=KQ_1$ given by the quiver $Q_1$
$$\hskip .5cm 
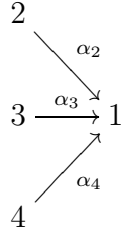
\xymatrix{
&{2}\ar [dr]^{\alpha_{2}}\\
&{3}\ar[r]^{\alpha_{3}}&{1}\\
&{4}\ar [ur]_{\alpha_{4}}}\hskip .5cm$$
which is a subquiver of $Q$, is a subalgebra of $\Lambda$. There is  an indecomposable representation $M$ of $Q_1$ such that $\mathbf{dim}M=[2,1,1,1]^t$. $M$ is not local and by \cite[proposition 2.8]{NS}, $M$ is a $5$-factor serial right $\Lambda_1$-module. Then there is a $5$-factor serial right $\Lambda$-module which is a contradiction.
\item Case $2$: The algebra $\Lambda_2=\frac{KQ_2}{I_2}$ given by the quiver $Q_2$
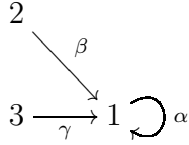
\begin{center}
$\begin{matrix}\xymatrix{{2}\ar[dr]^{\beta}\\
{3}\ar [r]_{\gamma}&{1}\ar@(ur,dr)^{\alpha}
}
\end{matrix}\hskip.5cm$
\end{center}
which is a subquiver of $Q$ and $I_2$ is a restriction of I to $Q_2$, is a subalgebra of $\Lambda$. There is an indecomposable representation $M$ of $(Q_2,I_2)$ such that $\mathbf{dim}M=[3,1,1]^t$. $M$ is not local and by \cite[proposition 2.8]{NS}, $M$ is a $5$-factor serial right $\Lambda_2$-module. Therefore  there is a $5$-factor serial right $\Lambda$-module which is a contradiction.
\end{itemize}
Now we show that for any $\alpha\in Q_1$, there is at most one arrow $\beta$ and at most one arrow $\gamma$ such that $\alpha\beta$ and $\gamma\alpha$ are not in $I$. Now  assume that there exist $\alpha, \beta_{1}, \beta_{2}\in Q_1$ such that $\alpha\beta_1$ and $\alpha\beta_2$ are not in $I$. Then we have two cases.
 \begin{itemize}
 \item Case $1$: The algebra $\Lambda_1=KQ_1$ given by the quiver $Q_1$
 $$\hskip .5cm 
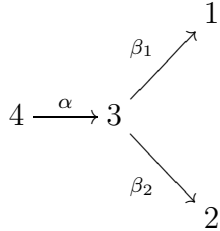
\xymatrix{
&&{1}\ar @{<-}[dl]_{\beta_1}\\
{4}\ar [r]^{\alpha}&{3}\ar [dr]_{\beta_2}\\
&&{2}}\hskip .5cm$$
which is a subquiver of $Q$, is a subalgebra of $\Lambda$. There is an indecomposable representation $M$ of $Q_1$ such that $\mathbf{dim}M=[1,1,2,1]^t$ that $M$ is not local and by \cite[proposition 2.8]{NS},  $M$ is a $5$-factor serial right $\Lambda_1$-module.  Therefore  there is a $5$-factor serial right $\Lambda$-module which is a contradiction.
  \item Case $2$: The algebra $\Lambda_2=\frac{KQ_2}{I_2}$ given by the quiver $Q_2$
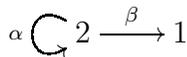
\begin{center}
$\begin{matrix}\xymatrix{{2}\ar@(ul,dl)_{\alpha} \ar[r]^{\beta}&{1}}
\end{matrix}\hskip.5cm$
\end{center}
which is a subquiver of $Q$ and $I_2$ is a restriction of $I$ to $Q_2$, is a subalgebra of $\Lambda$. Since $R^n\subseteq I$ for some $n\geq 3$, $\alpha^{n}\in I$ and $\alpha^{n-1}\beta\in I$. Then there is an indecomposable representation $M$ of $(Q_2, I_2)$ such that $\mathbf{dim}M=[2,4]^t$. $M$ is not local and by \cite[proposition 2.8]{NS}, $M$ is a $6$-factor serial right $\Lambda_2$-module. Therefore there is a $6$-factor serial right $\Lambda$-module which is a contradiction.\\
Now assume that there exist arrows $\alpha, \gamma_1, \gamma_2 \in Q_1$ such that $\gamma_1\alpha$ and $\gamma_2\alpha$ are not in $I$. Then we have two cases.
\item Case $1$: The algebra $\Lambda_3=KQ_3$ given by the quiver $Q_3$
 $$\hskip .5cm 
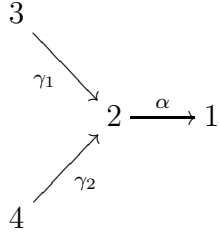
\xymatrix{
{3}\ar [dr]_{\gamma_1}\\
&{2}\ar [r]^{\alpha}&{1} \\
{4}\ar [ur]_{\gamma_2}}\hskip .5cm$$
which is a subquiver of $Q$, is a subalgebra of $\Lambda$. There is an indecomposable representation $M$ of $Q_3$ such that $\mathbf{dim}M=[1,2,1,1]^t$,  $M$ is not local and by \cite[proposition 2.8]{NS}, $M$ is a $5$-factor serial right $\Lambda_3$-module. Therefore there is a $5$-factor serial right $\Lambda$-module which is a contradiction.

\item Case $2$: The quiver $Q_4$ given by
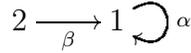
\begin{center}
$\begin{matrix} \xymatrix{{2}\ar[r]_{\beta}&{1}\ar@(ur,dr)^{\alpha}}
\end{matrix}\hskip.5cm$
\end{center}
 is a subquiver of $Q$. Let $I_4$ be the restriction of $I$ to $Q_4$. Then $\Lambda_4=KQ_4/I_4$ is a subalgebra of $\Lambda$. Since $R^n\subseteq I$ for some $n\geq 3$, $\alpha^{n}\in I$ and $\beta\alpha^{n-1}\in I$. There is an indecomposable representation $M$ of $(Q_4, I_4)$ such that $\mathbf{dim}M=[4,2]^t$, $M$ is not local and by \cite[proposition 2.8]{NS}, $M$ is a $6$-factor serial right $\Lambda_4$-module. Therefore, there is a $6$-factor serial right $\Lambda$-module which is a contradiction.\\
\end{itemize}
 \end{proof}

\begin{theorem}\label{T4}
Let $ \Lambda=\frac{KQ}{I} $ be a basic and connected  finite dimensional $K$-algebra. Then
$ \Lambda $ is a right $3$-Nakayama algebra if and only if $ \Lambda $ is a special biserial algebra of finite type that $\left( Q, I\right) $ satisfying the following conditions:
\begin{itemize}
\item[$(i)$] If there exist a walk $w$ and two different arrows $w_1$ and $w_2$ with the same target such that  $w_1^{+1}w_2^{-1}$ is a subwalk of $w$, then $w=w_1^{+1}w_2^{-1}$.
\item[$(ii)$] If there exist a walk $w$ and two different arrows $w_1$ and $w_2$ with the same source such that  $w_1^{-1}w_2^{+1}$ is a subwalk of $w$, then $length(w)\leq 3$.
\item[$(iii)$] If  there exist two paths $p$ and $q$ with the same target and the same source such that  $p-q\in I$, then $length(p)=length(q)=2$.
\item[$(iv)$]  At least one of the following conditions holds.
\begin{itemize}
\item[$(a)$] There exists a vertex $a$ of $Q_0$ such that, $|a^-|=2$.
\item[$(b)$] There exist a walk $w$ of length $3$ and two different arrows $w_1$ and $w_2$ with the same source such that  $w_1^{-1}w_2^{+1}$ is a subwalk of $w$.
\item[$(c)$]  There exist two paths $p$ and $q$ with the same target and the same source such that $p-q\in I$ and $length(p)=length(q)=2$.
\end{itemize}

\end{itemize}
\end{theorem}
\begin{proof}
Assume that $\Lambda$ is a right $3$-Nakayama algebra.
By Proposition \ref{p1},  $ \Lambda $ is a special biserial algebra of finite  type.
 Assume that the  condition $(i)$ does not hold. Then there exists a walk $w$ of length greater than or equal to $3$, such that $w$ has a subwalk of the form $w_1^{+1}w_2^{-1}$. Since $\Lambda$ is an algebra of finite type, the walk $w_1^{+1}w_2^{-1}$ has one of the following forms:
 \begin{itemize}

 \item First case: The walk $w_1^{+1}w_2^{-1}$ is of the form
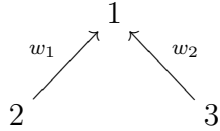
\begin{center}
$$\hskip .5cm \xymatrix{
&{1}\ar @{<-}[dl]_{w_1} \ar @{<-} [dr]^{w_2}\\
{2} &&{3}
}\hskip .5cm$$
\end{center}
 In this case $w$ has a subwalk of one of the following forms:
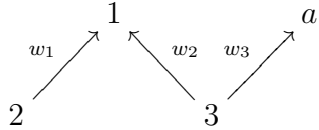
\begin{itemize}
\item[$(i)$]
$$\hskip .5cm \xymatrix{
&{1}\ar @{<-}[dl]_{w_1} \ar @{<-} [dr]^{w_2}&&{a}\ar @{<-}[dl]_{w_3}\\
{2} &&{3}
}\hskip .5cm$$

In this case the vertex $a$ can be either $2$ or $3$ or $4$.
\item[$(ii)$]
$$\hskip .5cm 
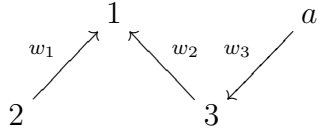
\xymatrix{
&{1}\ar @{<-}[dl]_{w_1} \ar @{<-} [dr]^{w_2}&&{a}\ar [dl]_{w_3}\\
{2} &&{3}
}\hskip .5cm$$

In this case the vertex $a$ can be either $1$ or $2$  or $3$ or $4$.
 \end{itemize}
  \item Second case: The walk $w_1^{+1}w_2^{-1}$ is of the form
$$\hskip .5cm 
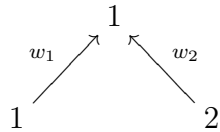
\xymatrix{
&{1}\ar @{<-}[dl]_{w_1} \ar @{<-} [dr]^{w_2}\\
{1} &&{2}
}\hskip .5cm$$
In this case $w$ has a subwalk of one of the following forms:
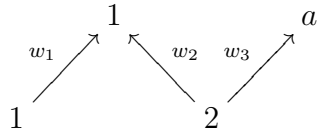
\begin{itemize}
\item[$(i)$]
$$\hskip .5cm \xymatrix{
&{1}\ar @{<-}[dl]_{w_1} \ar @{<-} [dr]^{w_2}&&{a}\ar @{<-}[dl]_{w_3}\\
{1} &&{2}
}\hskip .5cm$$
In this case the vertex $a$ can be either $2$ or $3$.
 \item[$(ii)$]
$$\hskip .5cm 
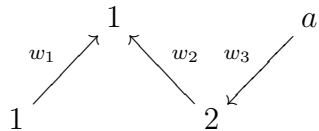
\xymatrix{
&{1}\ar @{<-}[dl]_{w_1} \ar @{<-} [dr]^{w_2}&&{a}\ar [dl]_{w_3}\\
{1} &&{2}
}\hskip .5cm$$
In this case the vertex $a$ can be either $1$ or  $2$  or $3$.
\item[$(iii)$]
$$\hskip .5cm 
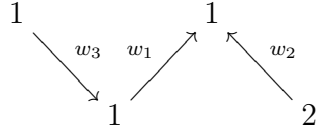
\xymatrix{
{1}\ar[dr]^{w_3}&&{1}\ar @{<-}[dl]_{w_1} \ar @{<-} [dr]^{w_2}\\
&{1} &&{2}
}\hskip .5cm$$
\item[$(iv)$]
$$\hskip .5cm 
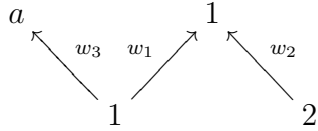
\xymatrix{
{a}\ar @{<-}[dr]^{w_3}&&{1}\ar @{<-}[dl]_{w_1} \ar @{<-} [dr]^{w_2}\\
&{1} &&{2}
}\hskip .5cm$$
In this case the vertex $a$ can be either $2$ or $3$.

In all the above cases, there is a non-local indecomposable right $\Lambda$-module of length $4$ that by \cite[proposition 2.8]{NS} is $4$-factor serial, which gives a contradiction.
\end{itemize}
 \end{itemize}

Now assume that the condition $(ii)$ does not hold. Then there exists a walk $w$ of length greater than or equal to $4$, such that $w$ has a subwalk of the form $w_1^{-1}w_2^{+1}$. Since $\Lambda$ is an algebra of finite type, the walk $w_1^{-1}w_2^{+1}$ has one of the following forms:
 \begin{itemize}
 \item First case: The walk $w_1^{-1}w_2^{+1}$ is of the form
$$\hskip .5cm 
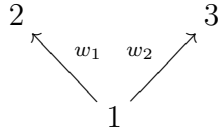
\xymatrix{
{2}\ar @{<-}[dr]^{w_1}&&{3}\ar @{<-}[dl]_{w_2}\\
&{1}
}\hskip .5cm$$
 In this case $w$ has a subwalk of one of the following forms:
 \begin{itemize}
 \item[$(i)$]
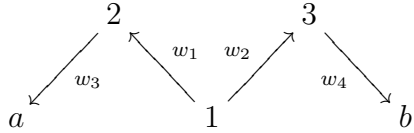
$$\hskip .5cm \xymatrix{
&{2}\ar @{<-}[dr]^{w_1}\ar[dl]^{w_3}&&{3}\ar @{<-}[dl]_{w_2}\ar[dr]_{w_4}&\\
{a}&&{1}&&{b}
}\hskip .5cm$$
In this case the vertices $a$ and $b$ can be either  $a=4$ and $b=5$ or $a=4$ and $b=1$.
\item[$(ii)$]
$$\hskip .5cm 
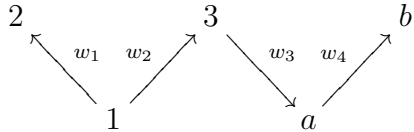
\xymatrix{
{2}\ar @{<-}[dr]^{w_1}&&{3}\ar @{<-}[dl]_{w_2}\ar[dr]^{w_3}&&{b}\\
&{1}&&{a}\ar[ur]^{w_4}&
}\hskip .5cm$$
In this case the vertices $a$ and $b$ can be either $a=4$ and $b=5$ or $a=4$ and $b=1$.
 \end{itemize}
 \item Second case: The walk $w_1^{-1}w_2^{+1}$ is of the form
 \begin{center}
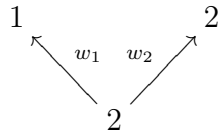
$$\hskip .5cm \xymatrix{
{1}\ar @{<-}[dr]^{w_1}&&{2}\ar @{<-}[dl]_{w_2}\\
&{2}
}\hskip .5cm$$
\end{center}
 In this case $w$ has a subwalk of one of the following forms:
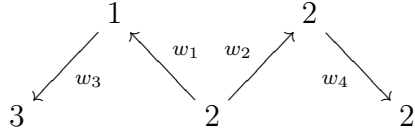
\begin{itemize}
\item[$(i)$]
$$\hskip .5cm \xymatrix{
&{1}\ar @{<-}[dr]^{w_1}\ar[dl]^{w_3}&&{2}\ar @{<-}[dl]_{w_2}\ar[dr]_{w_4}&\\
{3}&&{2}&&{2}
}\hskip .5cm$$

\item[$(ii)$]
$$\hskip .5cm 
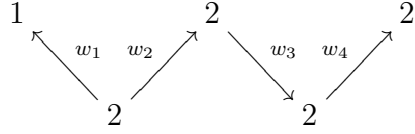
\xymatrix{
{1}\ar @{<-}[dr]^{w_1}&&{2}\ar @{<-}[dl]_{w_2}\ar[dr]^{w_3}&&{2}\\
&{2}&&{2}\ar[ur]^{w_4}&
}\hskip .5cm$$

\item[$(iii)$]
$$\hskip .5cm 
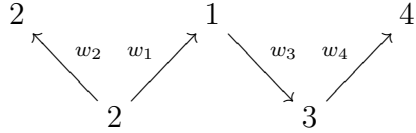
\xymatrix{
{2}\ar @{<-}[dr]^{w_2}&&{1}\ar @{<-}[dl]_{w_1}\ar[dr]^{w_3}&&{4}\\
&{2}&&{3}\ar[ur]^{w_4}&
}\hskip .5cm$$

In all the above cases, there is a $4$-factor serial indecomposable right $\Lambda$-module of length $5$, which gives a contradiction.
\end{itemize}
 \end{itemize}

Assume that the condition $(iii)$ does not hold. Then there exist two paths  $p=p_1...p_l$ and $q=q_1...q_r$ such that $p_i, q_j\in Q_1$, $s(p_1)=s(q_1)$, $t(p_l)=t(q_r)$, $p-q\in I$ and $l\geq 3$.
 $$\hskip .5cm 
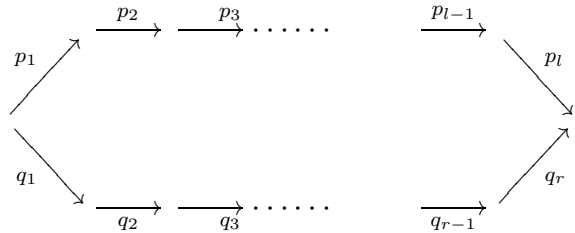
\xymatrix{
&{}\ar @{<-}[dl]_{p_{1}}\ar [r]^{p_{2}}&{}\ar[r]^{p_{3}}&\cdots\cdots&\ar[r]^{p_{l-1}}&\ar[dr]^{p_{l}}\\
{}\ar[dr]_{q_1}&&&&&&{}\ar @{<-}[dl]^{q_r}\\
&{}\ar [r]_{q_{2}}&{}\ar[r]_{q_3}&\cdots\cdots&\ar[r]_{q_{r-1}}&}\hskip .5cm$$
Then the string $w=p_{l-1}^{+1}p_{l}^{+1}q_{r}^{-1}\in \widetilde{\mathcal{S}}(\Lambda)$. $M(w)$ is a $4$-factor serial right $\Lambda$-module which gives a contradiction.
Now assume that the condition $(iv)$ does not hold. Then by \cite[Theorem 5.13] {NS}  $\Lambda$ is a right $t$-Nakayama algebra for some $t\leq 2$ which is a contradiction.
Conversely, assume that $(Q,I)$ satisfies the conditions $(i)$-$(v)$. By \cite{BR}, every indecomposable right $\Lambda$-module is either string or band or non-uniserial projective-injective. Since $\Lambda$ is representation finite, $\mathcal{B}(\Lambda)=\varnothing$. The conditions $(i)$, $(ii)$ and $(iii)$ imply that for any $w\in \widetilde{\mathcal{S}}(\Lambda)$, $w$ is either $w_1^{+1}...w_n^{+1}$ or $w_1^{-1}w_2^{+1}$ or $ w_1^{+1}w_2^{-1} $ or $ w_1^{-1}w_{2} ^{-1}w_3^{+1}$. If $w=w_1^{+1}...w_n^{+1}$, then $M(w)$ is uniserial. If $w=w_1^{-1}w_2^{+1}$, then $M(w)$ is $2$-factor serial. If $w=w_1^{+1}w_2^{-1}$ or $w=w_1^{-1}w_{2} ^{-1}w_3^{+1}$, then $M(w)$ is $3$-factor serial. By the condition $(iii)$, if there exists a non-uniserial projective-injective right $\Lambda$-module $M$, then $M$ is $3$-factor serial. The condition $(v)$ implies that, there exists at least one string module $M(w)$, where either $w=w_1^{-1}w_{2} ^{-1}w_3^{+1}$ or $w_1^{+1}w_2^{-1} $. Thus there exists a $3$-factor serial right $\Lambda$-module. Therefore $\Lambda$ is right $3$-Nakayama and the result follows.
\end{proof}

 \begin{remark}
  If the condition $(iii)$ of the Theorem \ref{T4} holds, then there exists a non-uniserial projective-injective $3$-factor serial right $\Lambda$-module.
 \end{remark}

\maketitle

\section{self-injective special biserial algebras of finite type}

In this section, we first characterize self-injective finite dimensional special biserial algebras of finite type. Then we give a characterization of right $3$-Nakayama self-injective algebras.

 \begin{theorem} \label{T10}
 Let $\Lambda=\frac{KQ}{I}$ be a basic and connected finite dimensional $K$-algebra.  Then $\Lambda$ is non-Nakayama self-injective special biserial algebra of finite type if and only if $\Lambda$ is given by the quiver $Q=Q_{m, n, s}$ with $s\geq 1$ and $m,n\geq 2$,
$$
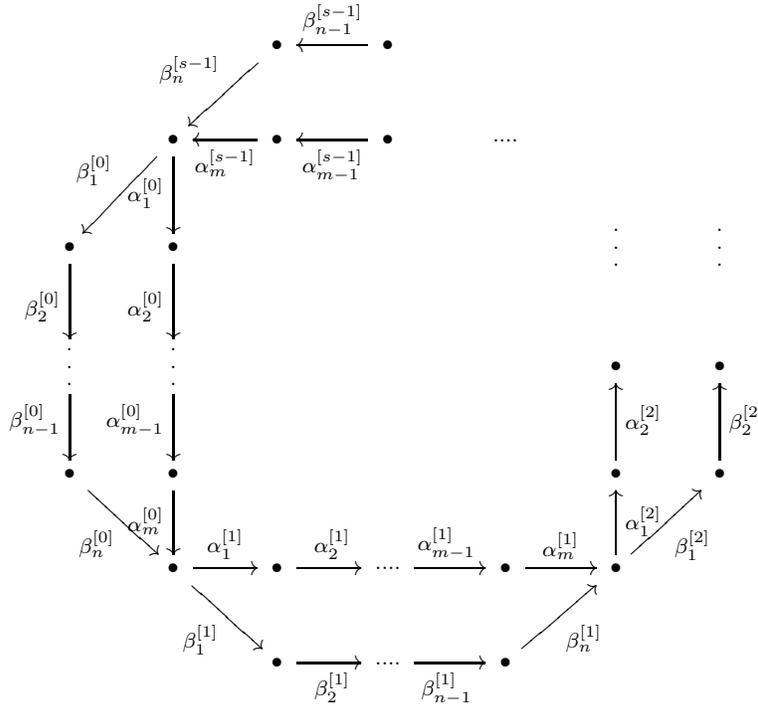
\xymatrix{
&& {\bsm\bullet \esm}\ar[dl]_{\beta_{n}^{[s-1]}}&{\bsm \bullet \esm}\ar[l]_{\beta_{n-1}^{[s-1]}}\\
&{\bsm\bullet\esm}\ar[dl]_{\beta_{1}^{[0]}}\ar[d]_{\alpha_{1}^{[0]}}&{\bsm\bullet \esm}\ar[l]^{\alpha_{m}^{[s-1]}} &{\bsm\bullet\esm}\ar[l]^{\alpha_{m-1}^{[s-1]}}&{\bsm....\esm} \\
{\bsm\bullet\esm}\ar[d]_{\beta_{2}^{[0]}}&{\bsm\bullet\esm}\ar[d]_{\alpha_{2}^{[0]}}&&&&{\bsm.\\.\\.\esm}&{\bsm.\\.\\.\esm} \\
{\bsm.\\.\\.\esm}\ar[d]_{\beta_{n-1}^{[0]}}&{\bsm.\\.\\.\esm}\ar[d]_{\alpha_{m-1}^{[0]}}&&&&{\bsm\bullet\esm}&{\bsm\bullet\esm}\\
{\bsm\bullet\esm}\ar[dr]_{\beta_{n}^{[0]}}&{\bsm\bullet\esm}\ar[d]_{\alpha_{m}^{[0]}}&&&&{\bsm\bullet\esm}\ar[u]_{\alpha_{2}^{[2]}}&{\bsm\bullet\esm}\ar[u]_{\beta_{2}^{[2]}} \\
&{\bsm\bullet\esm}\ar[r]^{\alpha_{1}^{[1]}}\ar[dr]_{\beta_{1}^{[1]}}&{\bsm\bullet\esm}\ar[r]^{\alpha_{2}^{[1]}}&{\bsm....\esm}\ar[r]^{\alpha_{m-1}^{[1]}}&{\bsm\bullet\esm}\ar[r]^{\alpha_{m}^{[1]}}&{\bsm\bullet\esm}\ar[u]_{\alpha_{1}^{[2]}} \ar[ur]_{\beta_{1}^{[2]}}\\
&&{\bsm \bullet \esm}\ar[r]_{\beta_{2}^{[1]}}&{\bsm....\esm}\ar[r]_{\beta_{n-1}^{[1]}}&{\bsm\bullet\esm}\ar[ur]_{\beta_{n}^{[1]}} \\}
$$
bounded by the following relations $R_{m, n, s}$:
 \begin{itemize}
 \item[$(i)$] $\alpha_{1}^{[i]}\cdots\alpha_{m}^{[i]}=\beta_{1}^{[i]}\cdots\beta_{n}^{[i]}$ for all $i\in \lbrace 0,\cdots,s-1\rbrace$;
  \item[$(ii)$] $\beta_n^{[i]}\alpha_{1}^{[i+1]}=0$, $\alpha_m^{[i]}\beta_{1}^{[i+1]}=0$ for all $i\in \lbrace 0,\cdots,s-2\rbrace$,  $\beta_{n}^{[s-1]} \alpha_{1}^{[0]}=0$ and $\alpha_m^{[s-1]}\beta_{1}^{[0]}=0$;
 \item[$(iii)$]
 \begin{itemize}
\item[$(a)$] Paths of the form $\alpha_i^{[j]}...\alpha_h^{[f]}$ of  length $m+1$ are equal to $0$;
\item[$(b)$] Paths of the form $\beta_i^{[j]}...\beta_h^{[f]}$ of length $n+1$ are equal to $0$.
 \end{itemize}
 \end{itemize}
 \end{theorem}

\begin{proof}
It is easy to see that $\Lambda=\frac{KQ}{I}$, where $Q=Q_{m, n, s}$, $I$ is an ideal generated by the relations $(i), (ii)$ and $(iii)$, $m,n\geq 2$ and $s\geq 1$ is a non-Nakayama self-injective special biserial algebra of finite type. Let $\Lambda=\frac{KQ}{I}$ be a non-Nakayama self-injective special biserial algebra of finite type, we show that $Q=Q_{m, n, s}$ with $s\geq 1$ and $m,n\geq 2$ bounded by relations $(i), (ii)$ and $(iii)$. \\
Since $\Lambda$ is self-injective, $Q$ has no sources and no sinks. Since $\Lambda$ is special biserial self-injective and non-Nakayama, then there exists $b\in Q_0$ such that $\mid b^+\mid=2$.  We show that for every vertex $a$ of $Q$, $|a^+|=|a^-|$. Assume on  contrary there exists a vertex $a$ such that $|a^-|=2$ and $|a^+|=1$. The following quiver is a subquiver of $Q$.
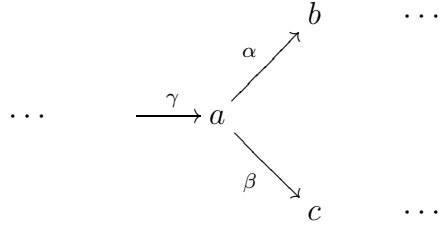
   $$\hskip .5cm \xymatrix{
&&&{b}\ar @{<-}[dl]_{\alpha}&\cdots\\
\cdots&{}\ar[r]^{\gamma}&{a}\ar[dr]_{\beta}\\
&&&{c}&\cdots&}\hskip .5cm$$
Since $\Lambda$ is special biserial then either $\gamma\alpha\in I$ or $\gamma\beta\in I$. If $\gamma\alpha\in I$, then the indecomposable injective right $\Lambda$-module $I(b)$ is not projective and if $\gamma\beta\in I$, then the indecomposable injective right $\Lambda$-module $I(c)$ is not projective which is a contradiction. The same argument shows that there is no vertex $a\in Q_0$ such that $|a^+|=2$ and $|a^-|=1$.
Consider a vertex $a$ of quiver $Q$ such that $|a^+|=|a^-|=2$. The following quiver is a subquiver of $Q$.
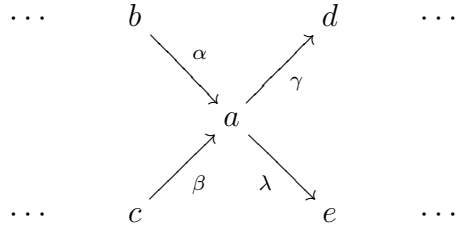
 $$\hskip .5cm \xymatrix{
\cdots&{b}\ar[dr]^{\alpha}&&{d}\ar @{<-}[dl]^{\gamma}&\cdots \\
&&{a}\ar @{<-}[dl]^{\beta}\ar[dr]_{\lambda}\\
\cdots&{c}&&{e}&\cdots}\hskip .5cm$$
We show that for the arrow $\gamma$, exactly one of the paths $\alpha\gamma$ or $\beta\gamma$ is in $I$ and for the arrow $\alpha$, exactly one of the paths $\alpha\gamma$ or $\alpha\delta$ is in $I$. If $\alpha\gamma\in I$ and $\beta\gamma\in I$, then the indecomposable injective right $\Lambda$-module $I(d)$ is not projective which gives a contradiction. If $\alpha\gamma\in I$, the same argument shows that $\alpha\lambda\not\in I$ and so $\beta\lambda\in I$ and $\beta\gamma\not\in I$. If $\beta\gamma\in I$, then $\alpha\gamma\not\in I$ and so $\alpha\lambda\in I$ and $\beta\lambda\not\in I$.
For any subquiver $Q'$ of $Q$ of the form

 $$\hskip .5cm 
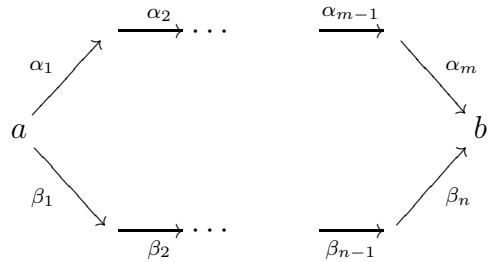
\xymatrix{
&{}\ar @{<-}[dl]_{\alpha_1}\ar [r]^{\alpha_2}&\cdots&\ar[r]^{\alpha_{m-1}}&\ar[dr]^{\alpha_{m}}\\
{a}\ar[dr]_{\beta_1}&&&&&{b}\ar @{<-}[dl]^{\beta_n}\\
&{}\ar [r]_{\beta_2}&\cdots&\ar[r]_{\beta_{n-1}}&}\hskip .5cm$$
Since $\Lambda$ is of finite type, $n\geq 2$ or $m\geq 2$. Now we show that both $m$ and $n$ are grater than or equal to $2$. Assume that $Q$ has a subquiver of the form
 \begin{center}
$
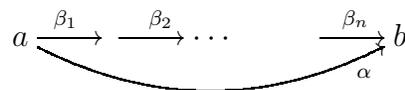
\begin{matrix}\xymatrix{{a} \ar@/_20pt/[rrrr]_>>>{\alpha} \ar[r]^{\beta_1}&{} \ar[r]^{\beta_2}&\cdots&{} \ar[r]^{\beta_{n}}&{b}\\
}
\end{matrix}\hskip.5cm$
\end{center}
 for some $n\geq 2$. If for some $i$, $\beta_1\cdots \beta_i\in I$, then the indecomposable injective right $\Lambda$-module $I(b)$ is not projective, which gives a contradiction and if $\beta_1\cdots \beta_n\not\in I$, then $\Lambda$ is representation infinite which gives a contradiction. Therefore $m\geq 2$.\\

For any subquiver $Q'$ of $Q$ of the form
 $$\hskip .5cm 
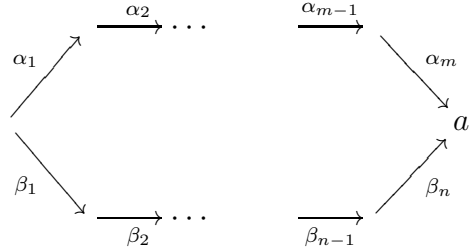
\xymatrix{
&{}\ar @{<-}[dl]_{\alpha_1}\ar [r]^{\alpha_2}&\cdots&\ar[r]^{\alpha_{m-1}}&\ar[dr]^{\alpha_{m}}\\
{}\ar[dr]_{\beta_1}&&&&&{a}\ar @{<-}[dl]^{\beta_n}\\
&{}\ar [r]_{\beta_2}&\cdots&\ar[r]_{\beta_{n-1}}&}\hskip .5cm$$

, with $m, n \geq 2$, we show that $\alpha_1...\alpha_m-\beta_1...\beta_n\in I$. Assume on the contrary that $\alpha_1...\alpha_m-\beta_1...\beta_n\notin I$. If either $\alpha_1...\alpha_i\in I$ for some $2\leq i\leq m$ or $(\beta_1...\beta_j\in I)$ for some $2\leq j\leq n$, then the indecomposable injective right $\Lambda$-module $I(a)$ is not projective which gives a contradiction. If there is no relation in this subquiver, then $\Lambda$ is not of finite representation type which gives a contradiction.\\
Assume that $Q$ has a subquiver of the form
 $$\hskip .5cm 
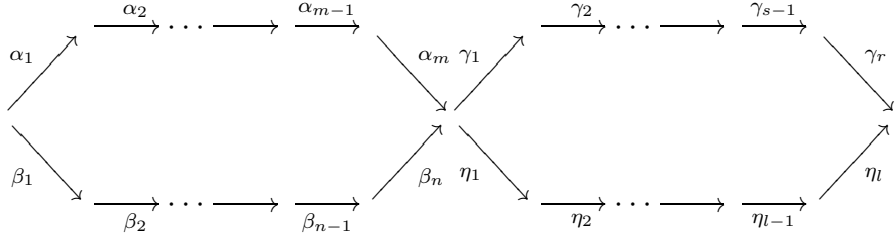
\xymatrix{
&{}\ar @{<-}[dl]_{\alpha_1}\ar [r]^{\alpha_2}&\cdots\ar[r]&\ar[r]^{\alpha_{m-1}}&\ar[dr]^{\alpha_{m}}&&{}\ar @{<-}[dl]_{\gamma_1}\ar[r]^{\gamma_2}&\cdots\ar[r]&{}\ar[r]^{\gamma_{s-1}}&{}\ar[dr]^{\gamma_r}\\
{}\ar[dr]_{\beta_1}&&&&&{}\ar @{<-}[dl]^{\beta_n}\ar[dr]_{\eta_1}&&&&&{}\ar @{<-}[dl]^{\eta_l}\\
&{}\ar [r]_{\beta_2}&\cdots\ar[r]&\ar[r]_{\beta_{n-1}}&&&{}\ar[r]_{\eta_2}&\cdots\ar[r]&{}\ar[r]_{\eta_{l-1}}&{}}\hskip .5cm$$
, with $m, n, r, l\geq 2$, bounded by relations $\alpha_1...\alpha_m - \beta_1...\beta_n\in I$, $\gamma_1...\gamma_r - \eta_1...\eta_l\in I $ and $\beta_n\gamma_1=\alpha_m\eta_1=0$. We show that in this case $r=m$ and $n=l$. Assume on a contrary that $m>r$. In this case there are two vertices $a$ and $b$ such that the indecomposable  projective right $\Lambda$-modules $P(a)$ and $P(b)$ have the same simple socle, which gives a contradiction.
 $$\hskip .5cm 
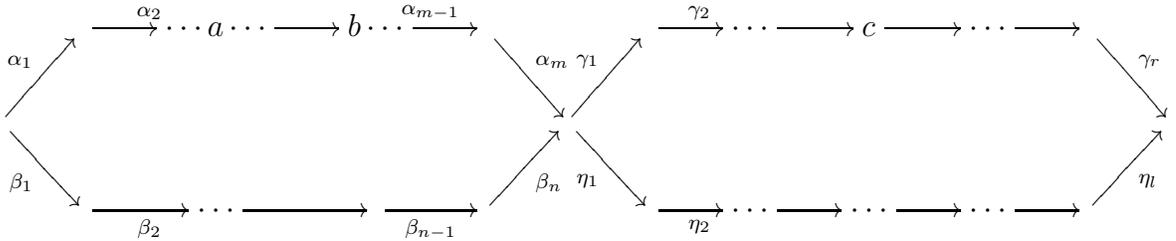
\xymatrix{
&{}\ar @{<-}[dl]_{\alpha_1}\ar [r]^{\alpha_2}&\cdots{a}\cdots\ar[r]&{b}\cdots\ar[r]^{\alpha_{m-1}}&\ar[dr]^{\alpha_{m}}&&{}\ar @{<-}[dl]_{\gamma_1}\ar[r]^{\gamma_2}&\cdots\ar[r]&{c}\ar[r]&\cdots\ar[r]&{}\ar[dr]^{\gamma_r}\\
{}\ar[dr]_{\beta_1}&&&&&{}\ar @{<-}[dl]^{\beta_n}\ar[dr]_{\eta_1}&&&&&&{}\ar @{<-}[dl]^{\eta_l}\\
&{}\ar [r]_{\beta_2}&\cdots\ar[r]&\ar[r]_{\beta_{n-1}}&&&{}\ar[r]_{\eta_2}&\cdots\ar[r]&\cdots\ar[r]&\cdots\ar[r]&{}}\hskip .5cm$$
If $m<r$, then there are two vertices $b$ and $c$ such that the indecomposable injective right $\Lambda$-modules $I(b)$ and $I(c)$ have the same simple top, which gives a contradiction.
$$\hskip .5cm 
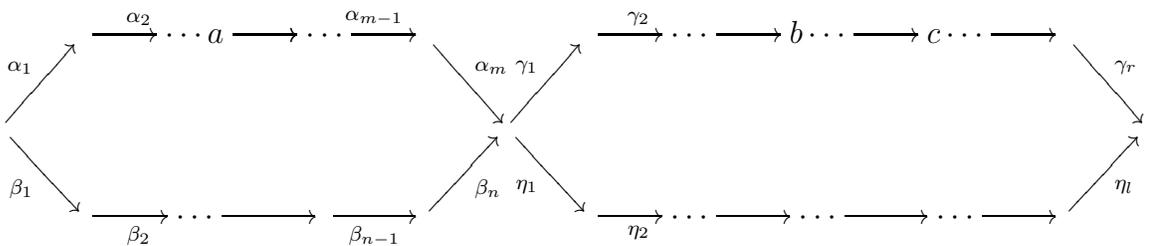
\xymatrix{
&{}\ar @{<-}[dl]_{\alpha_1}\ar [r]^{\alpha_2}&\cdots\ar[r]{a}\ar[r]&\cdots\ar[r]^{\alpha_{m-1}}&\ar[dr]^{\alpha_{m}}&&{}\ar @{<-}[dl]_{\gamma_1}\ar[r]^{\gamma_2}&\cdots\ar[r]&{b}\cdots\ar[r]&{c}\cdots\ar[r]&{}\ar[dr]^{\gamma_r}\\
{}\ar[dr]_{\beta_1}&&&&&{}\ar @{<-}[dl]^{\beta_n}\ar[dr]_{\eta_1}&&&&&&{}\ar @{<-}[dl]^{\eta_l}\\
&{}\ar [r]_{\beta_2}&\cdots\ar[r]&\ar[r]_{\beta_{n-1}}&&&{}\ar[r]_{\eta_2}&\cdots\ar[r]&\cdots\ar[r]&\cdots\ar[r]&{}}\hskip .5cm$$
The similar argument shows that $n=l$. Finally we show that any paths of form $\alpha_i^{[j]}...\alpha_h^{[f]}$ of length $m+1$ is zero. First we note that if there exist a positive integer $t$ and a path $w$ of the form $w=\alpha_i^{[j]}...\alpha_h^{[f]}$ of length $t$ such that $w=0$, then any path of the form $\alpha_i^{[j]}...\alpha_h^{[f]}$ of length $t$ should be zero. Since otherwise we can find an indecomposable projective right $\Lambda$-module, which is not injective. Now since by the above arguments $\alpha_1^{[i]}......\alpha_{m}^{[i]}-\beta_1^{[i]}....\beta_n^{[i]}=0$ and $\beta_n^{[i]}\alpha_1^{[i+1]}=0$, $\alpha_1^{[i]}......\alpha_{m}^{[i]}\alpha_1^{[i+1]}=0$. Therefor any paths of form $\alpha_i^{[j]}...\alpha_h^{[f]}$ of length $m+1$ is zero. The similar argument shows that any paths of the form $\beta_i^{[j]}...\beta_h^{[f]}$ of length $n+1$ is zero.
\end{proof}

The following Proposition provide a large class of self-injective right $m+n-1$-Nakayama algebras.

 \begin{proposition}\label{p8}
 Let $\Lambda=\frac{KQ}{I}$ be a basic and connected finite dimensional $K$-algebra such  that $Q= Q_{m, n, s}$ and $I=R_{m, n, s}$ with $s\geq 1$ and $m, n\geq 2$. Then $\Lambda$ is a right $(m+n-1)$-Nakayama algebra.
 \end{proposition}
 \begin{proof}
There exists a projective-injective non-uniserial right $\Lambda$-module $M$ of length $m+n$, such that for every indecomposable right $\Lambda$-module $N$, $l(M)\geq l(N)$ and $rad(M)$ is not local. Then by \cite[Corollary 2.8]{NS}, $M$ is $(m+n-1)$-factor serial. Therefor $\Lambda$ is right $(m+n-1)$-Nakayama.
 \end{proof}

 \begin{corollary}
 Let $\Lambda=\frac{KQ}{I}$ be a basic, connected and finite dimensional $K$-algebra. Then $\Lambda$ is right $3$-Nakayama self-injective if and only if $Q= Q_{2, 2, s}$ and $I=R_{2, 2, s}$.

 \end{corollary}
 \begin{proof}
 It follows from Proposition \ref{p8} and Theorem \ref{T10}
 \end{proof}

\section*{acknowledgements} The research of the first
author was in part supported by a grant from IPM (No. 96170419).


\end{document}